\documentclass[12pt]{article}
\usepackage[english]{babel}
\usepackage{amsmath}
\usepackage{graphicx}
\usepackage{amssymb}
\usepackage{amsthm}
\usepackage[colorinlistoftodos]{todonotes}
\usepackage{comment}
\usepackage{xcolor}
\usepackage{float}
\definecolor{green}{rgb}{0,.8,0} % This color needs to be adjusted to match the green in the picture

\newtheorem{theorem}{Theorem}[section] %[section] here insures the start of each section has sets the theorem counter back to 1. 
\newtheorem{lemma}[theorem]{Lemma}
\newtheorem{corollary}[theorem]{Corollary}

\newtheorem{proposition}[theorem]{Proposition}

\theoremstyle{definition}
\newtheorem*{defn}{Definition}
\newtheorem{example}[theorem]{Example}
\usepackage{graphicx}
\usepackage{caption}
\usepackage{subcaption}
\usepackage{tikz}
\tikzstyle{vertex}=[circle, draw, inner sep=0pt, minimum size=6pt]
\tikzstyle{rvertex}=[circle, red, fill, draw=black, inner sep=0pt, minimum size=6pt]
\tikzstyle{gvertex}=[circle, green, fill, draw=black, inner sep=0pt, minimum size=6pt]
\tikzstyle{bvertex}=[circle, blue, fill, draw=black, inner sep=0pt, minimum size=6pt]
\tikzstyle{Bvertex}=[circle, black, fill, draw=black, inner sep=0pt, minimum size=4pt]
\tikzstyle{pvertex}=[circle, purp, fill, draw=black, inner sep=0pt, minimum size=6pt]
\tikzstyle{overtex}=[circle, orange, fill, draw=black, inner sep=0pt, minimum size=6pt]

\newcommand{\Bvertex}{\node[Bvertex]}

\usetikzlibrary{arrows}
\newcommand{\midarrow}{\tikz \draw[- triangle 90] (0,0) -- +(.1,0);}
\usetikzlibrary{decorations.pathreplacing,decorations.markings}
\tikzset{arrow data/.style 2 args={%
      decoration={%
         markings,
         mark=at position #1 with \arrow{#2}},
         postaction=decorate}
      }
\title{Playing Games with Cacti}

\author{Samuel Adefiyiju\thanks{Raytheon Missiles and Defense, Portsmouth, RI 02871, USA (samueladefiyiju@gmail.com).} \and 
Heather Baranek\thanks{Department of Mathematics, University of Wisconsin-Eau Claire, Eau Claire, WI 54701, USA (baraneha5997@uwec.edu).} \and 
Abigail Daly\thanks{Department of Mathematics and Computer Science, Providence College, Providence, RI 02918, USA (adaly3@friars.providence.edu).} \and 
Xadia M. Goncalves \thanks{Department of Mathematics and Computer Science, Providence College, Providence, RI 02918, USA (xrodrigu@friars.providence.edu).} \and 
Mary Leah Karker \thanks{Department of Mathematics and Computer Science, Providence College, Providence, RI 02918, USA (mkarker@providence.edu).} \and 
Alison LaBarre \thanks{Newgrange Design, 607 North Ave Door 17, Wakefield, MA 01880, USA (aklabarre@outlook.com).} \and 
Shanise Walker \thanks{Department of Mathematics, University of Wisconsin-Eau Claire, Eau Claire, WI 54701, USA (walkersg@uwec.edu).}}

\date{}

\begin{document}
\maketitle

\begin{abstract}
The Game of Cycles is a two-player impartial mathematical game, introduced by Francis Su in his book \emph{Mathematics for Human Flourishing} (2020). The game is played on simple planar graphs in which players take turns marking edges using a sink-source rule. In Alvarado et al.,  the authors determine who is able to win on graphs with certain types of symmetry using a mirror-reverse strategy. In this paper, we analyze the game for specific types of cactus graphs using a modified version of the mirror-reverse strategy.
\end{abstract}


\section{Introduction}
The field of combinatorial game theory is a fertile area of mathematical study and graphs provide some of the most interesting playing grounds on which to explore these games. Well known games, such as Col and Snort~\cite{Conway}, originally contrived as map-coloring games, can be played and studied on game boards consisting of planar graphs. The Game of Cycles is an impartial combinatorial game introduced by Francis Su in his  book \emph{Mathematics for Human Flourishing}~\cite{Su}. In this game two players take turns marking edges on a planar graph. When playing a move each player is aware of every move that has been played up to that point in the game and there is no element of chance involved. Such a game is known as a sequential game with perfect information. This game is considered impartial because the moves available at any point in the game depend only on the current configuration of the board and not on the player whose turn it is to move. For more information on combinatorial games, see~\cite{Siegel}.

The Game of Cycles is played on any simple planar graph of vertices and edges like in Figure~\ref{fig-cell} below.

 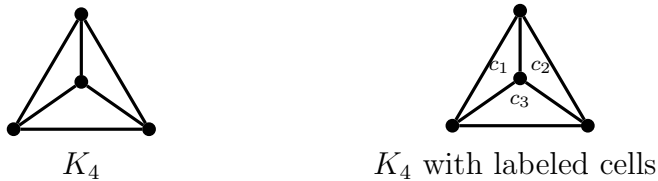
\begin{figure}[ht]\center
 \[
 \begin{array}{cc}	
\begin{tikzpicture}[scale=.9]
\begin{scope}[very thick, every node/.style={sloped,allow upside down}]
\Bvertex (A) at (0,0){};
\Bvertex (B) at (0,1) {};
\Bvertex (C) at (-1,-0.7) {};
\Bvertex (D) at (1,-0.7){};
  \draw (A) to(D);
  \draw (D)to (C);
  \draw (C)to  (A);
  \draw (B)to  (C);
  \draw (B) to(A);
  \draw (D)to (B);
\end{scope}
\end{tikzpicture}
\hspace{1in}&
\begin{tikzpicture}[scale=.9]
\begin{scope}[very thick, every node/.style={sloped,allow upside down}]
\Bvertex (A) at (0,0){};
\Bvertex (B) at (0,1) {};
\Bvertex (C) at (-1,-0.7) {};
\Bvertex (D) at (1,-0.7){};
\node(E) at (-0.31,-0.2)[label={[label distance=-2pt]: \scriptsize $c_1$}]{};
\node(F) at (0.31,-0.2)[label={[label distance=-2pt]: \scriptsize $c_2$}]{};
\node(G) at (0,-0.7)[label={[label distance=-2pt]: \scriptsize $c_3$}]{};
  \draw (A) to(D);
  \draw (D)to (C);
  \draw (C)to  (A);
  \draw (B)to  (C);
  \draw (B) to(A);
  \draw (D)to (B);
\end{scope}
\end{tikzpicture}\\
\hspace{-1in}K_4& K_4 \text{ with labeled cells }
\end{array}
\]
\caption{The complete graph $K_4$ game board and cells $c_1$, $c_2$, and $c_3$.}\label{fig-cell}
\end{figure}

We utilize the terminology defined originally by Alvarado et al.~\cite{Alvarado} to introduce the Game of Cycles. A planar graph of vertices and edges divides the plane into regions, which we call \emph{cells}, labelled $c_1$, $c_2$, and $c_3$  in  Figure~\ref{fig-cell}. The graph together with its bounded cells is the \emph{game board}.  Suppose there are two players and each player takes turns marking an edge that is not currently marked with an arrow, subject to the sink-source rule: players are not allowed to create a sink or a source, where a \emph{sink} is a vertex whose edges all point towards it and a \emph{source} is a vertex whose edges all point away from it. An example of a sink and a source is shown in Figure~\ref{fig-cells}. Therefore, every vertex must have at least one edge pointing towards it and at least one edge pointing away from it.  The object of the game is to be the first player to create a \emph{cycle cell}, which is a single cell in the game board whose whose boundary edges are all marked in either a clockwise or a counterclockwise direction. Figure~\ref{fig-cells} gives an example of a cycle cell. Creating a cycle cell is not always possible, so the first player to create a cycle cell or make the last possible move wins the game.

 \begin{figure}[ht]\center
\[
\begin{array}{ccc}	   
 \begin{tikzpicture}[scale=.9]
\begin{scope}[very thick, every node/.style={sloped,allow upside down}]
\Bvertex (A) at (0,0){};
\Bvertex (B) at (0,1) {};
\Bvertex (C) at (-1,-0.7) {};
\Bvertex (D) at (1,-0.7) {};
  \draw (A)-- node {\midarrow} (B);
  \draw (A)-- node {\midarrow} (C);
  \draw (A)-- node {\midarrow} (D);
\end{scope}
\end{tikzpicture}  
&
  \begin{tikzpicture}[scale=.9]
\begin{scope}[very thick, every node/.style={sloped,allow upside down}]
\Bvertex (A) at (0,0) {};
\Bvertex (B) at (0,1) {};
\Bvertex (C) at (-1,-0.7) {};
\Bvertex (D) at (1,-0.7) {};
  \draw (B)-- node {\midarrow} (A);
  \draw (C)-- node {\midarrow} (A);
  \draw (D)-- node {\midarrow} (A);
\end{scope}
\end{tikzpicture}  
&
\begin{tikzpicture}[scale=.9]
\begin{scope}[very thick, every node/.style={sloped,allow upside down}]
\Bvertex (A) at (0,0){};
\Bvertex (B) at (0,1) {};
\Bvertex (C) at (-1,-0.7) {};
\Bvertex (D) at (1,-0.7){};
  \draw (A)-- node {\midarrow} (D);
  \draw (D)-- node {\midarrow} (C);
  \draw (C)-- node {\midarrow} (A);
  \draw (B)to  (C);
  \draw (B) to(A);
  \draw (D)to (B);
\end{scope}
\end{tikzpicture}  \\
\text{ Source} &\text{ Sink} & \text{ Cycle cell}\\

\end{array}
\]
\caption{Examples of a source, sink, and cycle cell}\label{fig-cells}
\end{figure}
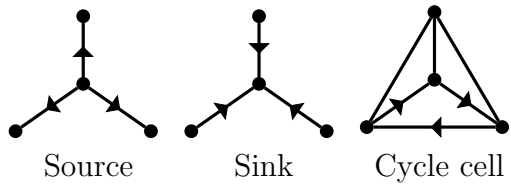
 
When an edge is marked, it can lead to consequences for other edges on the game board. A \emph{death move} occurs when an edge is marked with an arrow in such a way that it forms the penultimate arrow of a potential cycle cell. An edge is \emph{currently unplayable} if both possible markings of the edge lead to a either a sink/source or a death move. We say an edge is \emph{currently playable} if it is not currently unplayable.

There are also consequences for the sink-source rule. In particular, at degree 2 vertices the direction of the edge cannot change since changing directions will lead to a sink or source. A vertex in which all but one of its incident edges are pointed towards it is called an \emph{almost-sink}.  Thus, the last edge, if marked, must be marked with an arrow pointing away from it. Similarly, an \emph{almost-source} is a vertex in which all but one of its incident edges are pointing away from it. An edge is \emph{markable} if it can be marked with an arrow without violating the sink-source rule. However, if an edge is incident to two almost-sinks (or almost-sources), then the edge is \emph{unmarkable}. 

A \emph{strategy} is a sequence of moves for a player to make in any situation. A \emph{winning strategy} is a strategy which will force that player to win no matter how the other player plays. As a finite two-player sequential game with perfect information that is unable to end in a draw, Zermelo’s Theorem~\cite{Zermelo} tells us that for each game board one of the two players must possess a winning strategy. We say a player \emph{wins}  on a certain game board if the outcome of the game is not defined by a strategy and game play is deterministic. In studying the game of cycles, Alvarado et al.~\cite{Alvarado} explored winning strategies for various game boards and determined which player had a winning strategy or, in some cases, a predetermined win.

\section{Preliminary results}
In this section we summarize some results from Alvarado et al.~\cite{Alvarado} that are referenced in the proofs of our results. We begin with results for the cycle board on $n$ vertices. A cycle board on $n$ vertices, denoted $C_n$, is the cycle graph
with $n$ vertices and $n$ edges alternating along the boundary of the cell. 

\begin{theorem}{\rm \cite{Alvarado}} \label{cyclen}
Play on a $C_n$ is entirely determined by parity. For $n$ odd,  Player~1 wins and for $n$ even Player~2 wins.
\end{theorem}

The following lemma  was used to prove Theorem~\ref{cyclen}. 
\begin{lemma}{\rm \cite{Alvarado}} 
If a $C_n$ board has no markable edges, the number of unmarkable edges must be even. 
\end{lemma}

Game boards with involutive symmetry were also studied by the authors in~\cite{Alvarado}. A  game board has \emph{involutive symmetry} if there is a non-trivial symmetry of the board which is its own inverse. Every vertex $v$ (edge $e$) has a \emph{partner vertex} $v'$ (\emph{partner edge} $e$) which it is mapped to  under involutive symmetry. A vertex, edge, or cell is \emph{self-involutive} if the involution of that vertex, edge, or cell is itself. A cell is defined to be \emph{part-involutive} if it is not self-involutive, but one of the cell's edges has its partner also in the cell. A cell is \emph{nowhere-involutive} if no edge of the cell has its partner in the cell. For any board with involutive symmetry, every edge is partnered with another edge or itself.

Alvarado et al.~\cite{Alvarado} proved the following result for graphs with an involution. 

\begin{theorem}\label{invsym}{\rm \cite{Alvarado}} 
Let $G$ be a board with an involution such that each cell is either self-involutive or nowhere-involutive.  If there is no self-involutive edge, then Player~2 has a winning strategy. If there is exactly one self-involutive edge whose vertices are not fixed by the involution, then Player~1 has a winning strategy.
\end{theorem}

To prove Theorem~\ref{invsym}, the author's provided a  ``mirror-reverse" strategy for the winning player to use in responding to the other player. The mirror-reverse strategy is as follows:
\begin{itemize}
\item[(1)] If possible to win by completing a cycle, do so.
\item[(2)] If that is not possible, mirror the other player's strategy by observing the player's most recent move on an edge and playing the partner edge with its arrow reversed.
\end{itemize}

When there is exactly one self-involutive edge whose vertices are not fixed by the involution, Player~1 begins the game by marking the self-involutive edge on their first turn and then using the mirror-reverse strategy for all other subsequent moves. The next two results are special cases of Theorem~\ref{invsym}. 

\begin{corollary}\cite{Alvarado}\label{rotsym}
Let $G$ be a board with $180^{\circ}$ rotational symmetry, and no edge and its partner part of the same cell. If there is no edge through the center of the board then Player~2 has a winning strategy. If there is such an edge, then Player~1 has a winning strategy.
\end{corollary}

\begin{corollary}\label{mirrorsymbasic}\cite{Alvarado}
Let $G$ be a board that is symmetric by reflection across some line, with no edges along that axis of symmetry and at most one edge crossing that axis of symmetry.
On this board, Player~2 has a winning strategy if there is no edge crossing this axis of symmetry. If there is a single edge crossing this axis of symmetry, Player~1 has a winning strategy.
\end{corollary}

\section{Cactus graphs}\label{cactusgraph}
In this paper, we are concerned with winning strategies for games played on certain types of cactus graphs. A \emph{cactus} (sometimes called a cactus tree) is a connected graph in which any two simple cycles have at most one vertex in common. Equivalently, it is a connected graph in which every edge belongs to at most one simple cycle, or (for a nontrivial cactus) in which every block (maximal subgraph without a cut-vertex) is an edge or a cycle. We say that a graph is \emph{triangle-free} if it contains no $C_3$ graphs. 

The primary goal of this paper is to extend the symmetry argument of \cite{Alvarado} Theorem 2.3 to a certain class of graphs that do not possess global symmetry but instead can be viewed as consisting of many joined parts, each possessing its own  local ``axis of symmetry." This is done in Sections~\ref{cactusgraph} and~\ref{mainresultsection} with our main result being Theorem~\ref{mainresult}.
The results in these sections pertain to triangle-free cactus graphs in which every edge of the graph belongs to exactly one cycle. 
We demonstrate a winning strategy for
graphs of this type satisfying certain symmetry conditions. We will first work through examples of two, three, and four joined cycles in order to help motivate those conditions required in our main theorem in Section~\ref{mainresultsection}. The strategies outlined in the examples in Section~\ref{cactusgraph} are formally proved in the main result Theorem~\ref{mainresult}.

In Alvarado et al.~\cite{Alvarado}, the authors observed that their results showed for graphs with an odd number of edges, Player~1 had a winning strategy, otherwise Player~2 had a winning strategy. This led the authors to pose the following question: Is there a game board that does not follow the parity pattern of Player~1 having a winning strategy when there is an odd number of edges, otherwise Player~2 winning?
In Section 5, we discuss cactus graphs that are not triangle-free and use them to answer Alvarado et al.'s question by showing there are game boards with an even number of edges in which Player~1 has a winning strategy.

\subsection{Two joined cycles}

We begin by detailing the results of the class of triangle-free cactus graphs composed of two cycles joined together at a single vertex. Notice that when two cycles are joined together at a vertex the resulting graph has reflective symmetry with the axis of symmetry passing through the joining vertex and through either a vertex or an edge in each of the joined cycles, as shown in Figure~\ref{fig:symgraph} Board $A$ below. In the case that the two joined cycles are of the same degree, the board will also have an axis of reflective symmetry passing only through the joining vertex, as shown in Figure~\ref{fig:symgraph} Board $B$.

\begin{figure}[ht]
\[
\begin{array}{cc}	
\begin{tikzpicture}[scale=.45]
			 \begin{scope}[very thick, every node/.style={sloped,allow upside down}]
			 \Bvertex (A) at (1.5,0){};
			 \Bvertex (B) at (.5, 1.414){};
			 \Bvertex (C) at (-1.5,1.414){};
			 \Bvertex (D) at (-1.5,-1.414){};
			 \Bvertex (E) at (.5,-1.414){};
			 \Bvertex (F) at (2.586, 1.414){};
			 \Bvertex (G) at (4,1.414){};
			 \Bvertex (H) at (5.414,1.414){};
			 \Bvertex (I) at (5.414,-1.414){};
			 \Bvertex (J) at (4,-1.414){};
			 \Bvertex (K) at (2.586,-1.414){};
			 \Bvertex (L) at (-2.6, 0){};
			 \draw (A) to (B);
			 \draw (B) to (C);
			 \draw (L) to (C); 
			 \draw (L) to (D); 
			 \draw (E) to (A);
			 \draw (A) to (F);
			 \draw (F) to (G); 
			 \draw (G) to (H); 
			 \draw (H) to (I);
			 \draw (I) to (K);
			 \draw (J) to (K); 
			 \draw (K) to (A); 
			 \draw (D) to (E); 
			 \draw[dashed] (5.7,0) --(-2.9,0){};
			 \end{scope}
\end{tikzpicture}
&
\begin{tikzpicture}[scale=.45]
			 \begin{scope}[very thick, every node/.style={sloped,allow upside down}]
			 \Bvertex (A) at (1.5,0){};
			 \Bvertex (B) at (.414, 1.414){};
			 \Bvertex (C) at (-1,1.414){};
			 \Bvertex (D) at (-1,-1.414){};
			 \Bvertex (E) at (.414,-1.414){};
			 \Bvertex (F) at (2.586, 1.414){};
			 \Bvertex (G) at (4,1.414){};
			 \Bvertex (H) at (5.414,1.414){};
			 \Bvertex (I) at (5.414,-1.414){};
			 \Bvertex (J) at (4,-1.414){};
			 \Bvertex (K) at (2.586,-1.414){};
			 \Bvertex (L) at (-2.414,1.414){};
			 \Bvertex (M) at (-2.414,-1.414){};
			 \draw (A) to (B);
			 \draw (B) to (C);
			 \draw (E) to (A);
			 \draw (A) to (F);
			 \draw (F) to (G); 
			 \draw (G) to (H); 
			 \draw (H) to (I);
			 \draw (I) to (K);
			 \draw (J) to (K); 
			 \draw (K) to (A); 
			 \draw (D) to (E); 
			 \draw (C) to (L); 
			 \draw (L) to (M); 
			 \draw (M) to (D); 
			 \draw[dashed] (1.5,1.414) --(1.5,-1.414){};
			 \end{scope}
\end{tikzpicture}\\
\text{Board } A & \text{ Board }  B\\
\end{array}
\]
\caption{Two cycles joined at a vertex with their axis of reflective symmetry shown.}
\label{fig:symgraph}
\end{figure}
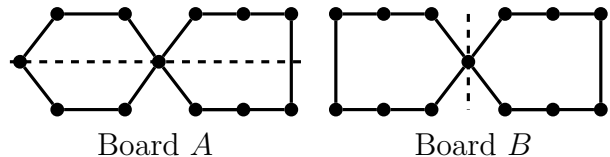

Because of this symmetry, a winning strategy in these cases follows from Corollary~\ref{mirrorsymbasic}. 

\begin{proposition}\label{parity}
Let $G$ be a board containing two cycles, $C_m$ and $C_n$, with $m>3$ and $n>3$, such that $C_m$ and $C_n$ are connected by a single vertex.  If $C_m$ and $C_n$ are of different parity, then Player~1 has a winning strategy. If $C_m$ and $C_n$ are both of  even parity, then Player~2 has a winning strategy. 
\end{proposition}

\begin{proof}
Note that $G$ has reflective symmetry across the line that passes through the degree 4 vertex and through either a vertex or edge in both $C_m$ and $C_n$, as displayed in Board A of Figure \ref{fig:symgraph}. When $m$ and $n$ are of different parity, the result follows directly from Corollary~\ref{mirrorsymbasic} where $G$ has exactly one self-involutive edge. When $m$ and $n$ are of even parity, there is no edge that crosses the axis of symmetry and the result again follows from  Corollary~\ref{mirrorsymbasic}.
\end{proof}

Recall that the proof of Theorem~\ref{invsym} uses the mirror-reverse strategy for the winning player. The mirror-reverse strategy only applies if there is no self-involutive edge or exactly one self-involutive edge whose vertices are not fixed by the involution. In the case where there are two odd cycles of different lengths, there are two self-involutive edges and the mirror-reverse strategy fails. 

\begin{example}[Two joined odd cycles]

Figure~\ref{fig:c57} shows an example of the mirror-reverse strategy failing and gives a modification to the strategy to show there is a winning strategy for Player~2. Note that the labelling on the edges denote the direction and the order in which the edges were marked with the odd numbered edges being marked by Player~1 and the even numbered edges being marked by Player~2. We use this labelling for all game boards played throughout the paper.
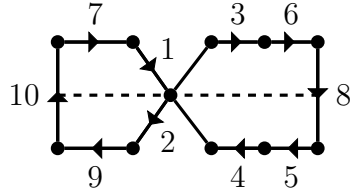
\begin{figure}[H]
    \centering
\begin{tikzpicture}[scale=.5]
			 \begin{scope}[very thick, every node/.style={sloped,allow upside down}]
			 \Bvertex (A) at (1.5,0){};
			 \Bvertex (B) at (.5, 1.414){};
			 \Bvertex (C) at (-1.5,1.414){};
			 \Bvertex (D) at (-1.5,-1.414){};
			 \Bvertex (E) at (.5,-1.414){};
			 \Bvertex (F) at (2.586, 1.414){};
			 \Bvertex (G) at (4,1.414){};
			 \Bvertex (H) at (5.414,1.414){};
			 \Bvertex (I) at (5.414,-1.414){};
			 \Bvertex (J) at (4,-1.414){};
			 \Bvertex (K) at (2.586,-1.414){};
			 \draw (A) to (B);
			 \draw (B) to (C);
			 \draw (D) to (C); 
			 \draw (E) to (A);
			 \draw (A) to (F);
			 \draw (F) to (G); 
			 \draw (G) to (H); 
			 \draw (H) to (I);
			 \draw (I) to (K);
			 \draw (J) to (K); 
			 \draw (K) to (A); 
			 \draw (D) to (E); 
			 \draw (B) -- node[font=\small, label={[label distance=-6pt]:$1$} ] {\midarrow} (A);
			 \draw (A) -- node[font=\small, label={[label distance=-6pt]:$2$} ] {\midarrow} (E);
			 \draw (F) -- node[font=\small, label={[label distance=-2pt]:$3$} ] {\midarrow} (G);
			 \draw (J) -- node[font=\small, label={[label distance=-2pt]:$4$} ] {\midarrow} (K);
			 \draw (I) -- node[font=\small, label={[label distance=-2pt]:$5$} ] {\midarrow} (J);
			 \draw (G) -- node[font=\small, label={[label distance=-2pt]:$6$} ] {\midarrow} (H);
			 \draw (C) -- node[font=\small, label={[label distance=-2pt]:$7$} ] {\midarrow} (B);
			 \draw (H) -- node[font=\small, label={[label distance=-2pt] above:$8$} ] {\midarrow} (I);
			 \draw (E) -- node[font=\small, label={[label distance=-2pt]:$9$} ] {\midarrow} (D);
			 \draw (D) -- node[font=\small, label={[label distance=-2pt]:$10$} ] {\midarrow} (C);
		     \draw[dashed] (-1.5,0) --(5.414,0){};
			 \end{scope}
\end{tikzpicture}

    \caption{Two odd cycles $C_5$ and $C_7$ joined at a vertex with reflective symmetry about the horizontal line.}
    \label{fig:c57}
\end{figure}

In Figure~\ref{fig:c57}, there is a game board consisting of two odd cycles joined at a single vertex, cycle $C_5$ on the left and cycle $C_7$ on the right. Player~1 begins the game by playing on the edge marked 1. Utilizing the mirror-reverse strategy, Player~2 plays the partner edge to Player~1's  first marked edge on their first turn in the opposite direction, and thus plays move 2. Moves 3 through 6 continues to follow the mirror-reverse strategy, with Player~2 utilizing the strategy. After Player~1 plays move 7, note that using the mirror-reverse strategy leads to a death move for Player~2 on the cycle graph $C_5$. The mirror-reverse strategy fails here for Player~2 because if they were to play the partner edge with the arrow reversed to follow the same direction  (clockwise or counterclockwise) then Player~1 could mark the self-involutive edge (the last unmarked edge on the cycle $C_5$) and win by completing a cycle. Instead of making a death move, Player~2 can modify the mirror-reverse strategy and play the self-involutive edge on the cycle $C_7$ in its playable direction, denoted as move 8. Note that any edge Player~1 plays on their next turn would be a death move. So suppose Player~1 marks the edge denoted 9 on the $C_5$ board. After move 9 is made, Player~2 wins by marking the edge denoted 10 and creating a cycle.

\end{example}

In order to address the odd-odd example above and also larger cactus graphs as in our main theorem we introduce a modified mirror-reverse strategy. 
The modified mirror-reverse strategy can be played on boards with involutive symmetry and is defined from the perspective of Player~2 because it relies on responding to the opposing player's moves. At the start of the game, if there are an even number of self-involutive edges on the game board, then Player~2 can implement the modified mirror-reverse strategy. However,  if the total number of self-involutive edges is odd, then Player~1 will begin the game by marking a self-involutive edge and then is able to implement the modified mirror-reverse strategy functioning as the second player.

\begin{defn}

We define the \emph{modified mirror-reverse strategy} as follows:
\begin{enumerate}
\item If possible  to win by completing a cycle, do so.
\item If it is not possible to win by completing a cycle and the opposing player plays a self-involutive edge, then  play an available self-involutive edge in any playable direction. 
\item If it is not possible to win by completing a cycle and the opposing player does not play a self-involutive edge, then  mirror the opposing player's strategy by  observing the player's most recent move on an edge and playing the partner edge with its arrow reversed unless that move is a death move. If the move is a death move,  play an available self-involutive edge in any playable direction. 
\end{enumerate}
\end{defn}

Utilizing this modified mirror reverse strategy, Player~2 has a winning strategy in the case of two joined odd cycles.

The following proposition is a consequence of the main result Theorem~\ref{mainresult} in Section~\ref{mainresultsection}.

\begin{proposition}\label{parity}
Let $G$ be a board containing two cycles, $C_m$ and $C_n$, with $m>3$ and $n>3$, such that $C_m$ and $C_n$ are connected by a single vertex.  If $C_m$ and $C_n$ are both of odd parity, then Player~2 has a winning strategy. 
\end{proposition}

\subsection{Three joined cycles}

Let's consider how we might generalize such a strategy to the case of cactus graphs consisting of three joined cycles. The first question that arises is: How should the axis of symmetry be defined for the graph? Consider the cactus graph in Figure~\ref{fig:c579} below.  Note that this graph does not have reflective symmetry. However, we can define axes of symmetry for each of the three individual cycles, which gives each cycle its own  local axis of symmetry. For the outer cycles $C_5$ and $C_7$, we will choose the axes of symmetry to pass through the vertices of degree 4. For the middle cycle $C_9$, an axis passing through the two degree 4 vertices would not define a reflective symmetry. Instead, we must define an axis of symmetry which is equidistant from the two degree 4 vertices, thus causing them to be partners via reflection, as shown in the figure. The partnering of the higher degree vertices is crucial to our strategy, as we will discuss later on. We define a vertex to be a vertex of \emph{high degree} if the vertex is of degree 4 or more. Note there are many ways to join three cycles including the board in which all three cycles are joined at a single vertex. In this case, we would define all three axes of symmetry to pass through the degree 6 vertex. In general, we utilize local axes of symmetry for individual cycles to define the reflective symmetry of the graph.

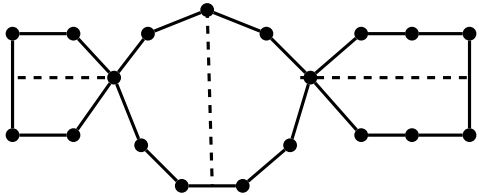
\begin{figure}[ht]
\centering
\begin{tikzpicture}[scale=.45 ]
			 \begin{scope}[very thick, every node/.style={sloped,allow upside down}]
			 	 \Bvertex (A) at (0,-1){};
			 \Bvertex (B) at (1.8, -1){};
			 \Bvertex (C) at (3.2,0.2){};
			 \Bvertex (D) at (3.8,2.2){};
			 \Bvertex (E) at (2.5,3.5){};
			 \Bvertex (F) at (.75,4.2){};
			 \Bvertex (G) at (-1,3.5){};
			 \Bvertex (H) at (-2,2.2){};
			 \Bvertex (I) at (-1.2,0.2){};
		     \Bvertex (J) at (-3.2,3.5){};
			 \Bvertex (K) at (-5,3.5){};
			 \Bvertex (L) at (-5,0.5){};
			 \Bvertex (M) at (-3.2,0.5){};
			 \Bvertex (N) at (5.3,3.5){};
			 \Bvertex (S) at (5.3,.5){};
			 \Bvertex (O) at (6.8,3.5){};
			 \Bvertex (R) at (6.8,0.5){};
			 \Bvertex (P) at (8.5,3.5){};
			 \Bvertex (Q) at (8.5,0.5){};
			 \draw (A) to (B);
			 \draw (B) to (C);
			 \draw (C) to (D); 
			 \draw (D) to (E);
			 \draw (E) to (F);
			 \draw (F) to (G); 
			 \draw (G) to (H); 
			 \draw (H) to (I);
			 \draw (I) to (A);
			 \draw (H) to (J); 
			 \draw (J) to (K); 
			 \draw (K) to (L); 
			 \draw (L) to (M);
			 \draw (M) to (H);
			 \draw (D) to (N); 
			 \draw (N) to (O);
			 \draw (O) to (P);
			 \draw (P) to (Q); 
			 \draw (Q) to (R); 
			 \draw (R) to (S);
			 \draw (S) to (D);
			 \draw[dashed] (.75,4.2) --(0.9,-1){};
			 \draw[dashed] (-1.8, 2.2) --(-5,2.2){};
			 \draw[dashed] (3.5,2.2) --(8.5,2.2){};
			 \end{scope}
\end{tikzpicture}

\caption{Cactus graph with three cycles}
\label{fig:c579}
\end{figure}

\begin{example}[Three joined cycles winning strategy]
We now utilize the modified mirror-reverse strategy and show that Player~1 has a winning strategy on the graph in Figure~\ref{fig:c579}. 
Note that in Figure~\ref{fig:c579}, there are three edges which are self-involutive with respect to the axes of symmetry on the cycles in which they lie.   Since there are an odd number of self-involutive edges, Player~1 should begin the game by marking any one of the three self-involutive edges. Suppose Player~1 marks the self-involutive edge on the cycle $C_9$. From this point on, Player~1 will use the modified mirror-reverse strategy to respond to Player~2's moves. 
Figure~\ref{fig:c579played} demonstrates an example of a completed game on the game board in Figure~\ref{fig:c579} using the modified mirror-reverse strategy.

\begin{figure}[H]
\centering
\begin{tikzpicture}[scale=.6]
			 \begin{scope}[very thick, every node/.style={sloped,allow upside down}]
			 \Bvertex (A) at (0,-1){};
			 \Bvertex (B) at (1.8, -1){};
			 \Bvertex (C) at (3.2,0.2){};
			 \Bvertex (D) at (3.8,2.2){};
			 \Bvertex (E) at (2.5,3.5){};
			 \Bvertex (F) at (.75,4.2){};
			 \Bvertex (G) at (-1,3.5){};
			 \Bvertex (H) at (-2,2.2){};
			 \Bvertex (I) at (-1.2,0.2){};
		     \Bvertex (J) at (-3.2,3.5){};
			 \Bvertex (K) at (-5,3.5){};
			 \Bvertex (L) at (-5,0.5){};
			 \Bvertex (M) at (-3.2,0.5){};
			 \Bvertex (N) at (5.3,3.5){};
			 \Bvertex (S) at (5.3,.5){};
			 \Bvertex (O) at (6.8,3.5){};
			 \Bvertex (R) at (6.8,0.5){};
			 \Bvertex (P) at (8.5,3.5){};
			 \Bvertex (Q) at (8.5,0.5){};
			 \draw (A) to (B);
			 \draw (B) to (C);
			 \draw (C) to (D); 
			 \draw (D) to (E);
			 \draw (E) to (F);
			 \draw (F) to (G); 
			 \draw (G) to (H); 
			 \draw (H) to (I);
			 \draw (I) to (A);
			 \draw (H) to (J); 
			 \draw (J) to (K); 
			 \draw (K) to (L); 
			 \draw (L) to (M);
			 \draw (M) to (H);
			 \draw (D) to (N); 
			 \draw (N) to (O);
			 \draw (O) to (P);
			 \draw (P) to (Q); 
			 \draw (Q) to (R); 
			 \draw (R) to (S);
			 \draw (S) to (D);
			 \draw[dashed] (.75,4.2) --(0.9,-1){};
			 \draw[dashed] (-1.8, 2.2) --(-5,2.2){};
			 \draw[dashed] (3.5,2.2) --(8.5,2.2){};
			 \draw (A) -- node[font=\small, label={[label distance=-2pt]below:$1$} ] {\midarrow} (B);
			 \draw (F) -- node[font=\small, label={[label distance=-5pt] below right:$2$} ] {\midarrow} (G);
			 \draw (E) -- node[font=\small, label={[label distance=-5pt] below left:$3$} ] {\midarrow} (F);
			 \draw (G) -- node[font=\small, label={[label distance=-5pt]:$4$} ] {\midarrow} (H);
			 \draw (D) -- node[font=\small, label={[label distance=-5pt]:$5$} ] {\midarrow} (E);
			 \draw (H) -- node[font=\small, label={[label distance=-6pt]:$6$} ] {\midarrow} (I);
			 \draw (C) -- node[font=\small, label={[label distance=-5pt]:$7$} ] {\midarrow} (D);
			 \draw (J) -- node[font=\small, label={[label distance=-7pt]:$8$} ] {\midarrow} (H);
			 \draw (H) -- node[font=\small, label={[label distance=-7pt]:$9$} ] {\midarrow} (M);
			 \draw (K) -- node[font=\small, label={[label distance=-5pt]:$10$} ] {\midarrow} (J);
			 \draw (P) -- node[font=\small, label={[label distance=-5pt]:$11$} ] {\midarrow} (Q);
			 \draw (O) -- node[font=\small, label={[label distance=-3pt]below:$12$} ] {\midarrow} (N);
			 \draw (S) -- node[font=\small, label={[label distance=-5pt]below:$13$} ] {\midarrow} (R);
			 \draw (N) -- node[font=\small, label={[label distance=-7pt]below:$14$} ] {\midarrow} (D);
			 \draw (D) -- node[font=\small, label={[label distance=-7pt]below:$15$} ] {\midarrow} (S);
			 \draw (I) -- node[font=\small, label={[label distance=-5pt]:$16$} ] {\midarrow} (A);
			 \draw (B) -- node[font=\small, label={[label distance=-5pt]:$17$} ] {\midarrow} (C);
			 \end{scope}
\end{tikzpicture}

\caption{Completed game on the game board in Figure~\ref{fig:c57}.}
\label{fig:c579played}
\end{figure}
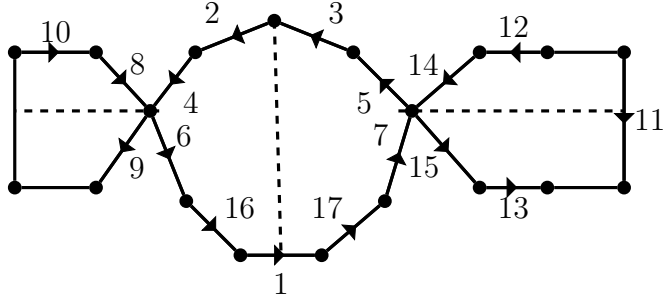
In the completed game shown in Figure~\ref{fig:c579played}, note that none of Player~1's moves on the cycle $C_9$ can result in a death move since a death move would only occur by marking the second-to-last edge on the cycle, which would be a Player~2 move by parity.

If Player~2 plays the self-involutive edge on either the cycle $C_5$ or the cycle $C_7$, Player~1 should respond by marking the remaining self-involutive edge. (This does not occur in our sample game.) If Player~2 plays on a non-self-involutive edge on either cycles  $C_5$ or $C_7$, then Player~1 should respond by marking the mirror-reverse, as long as this move is not a death move. This is demonstrated by moves 8 and 9 on the $C_5$ and moves 12 through 15 on the $C_7$. If the mirror-reverse were to be a death move, then Player~1 should instead mark the self-involutive edge on the other cycle. When Player~2 plays move 10 on the cycle $C_5$, we note that the mirror-reverse would be a death move for Player~1. So instead Player~1 plays move 11 on the self-involutive edge on the cycle $C_7$. If the death move had instead occurred on the $C_7$, Player~1 would have marked the self-involutive edge on the cycle $C_5$.

In our sample game, Player~2 is forced to make a death move after Player~1 plays move 15. Since sinks and sources are not permitted, the only markable edges remaining at this point in the game are either a death move on the cycle $C_5$ or on the cycle $C_9$ (as shown in the sample game). Here, Player~1 wins by completing a cycle on the $C_9$ graph. There are a number of other ways in which a game could play out on this board using the modified mirror-reverse strategy--- all of which will result in a Player~1 win. The reader is encouraged to explore some of these games on their own.

\end{example}

In Theorem~\ref{mainresult} (the main result), we show that in general, the winning strategy on triangle-free cactus graph game boards depends on the number of self-involutive edges. If there is an odd number of self-involutive edges, Player~1 has the winning strategy; if there is an even number of self-involutive edges, Player~2 has the winning strategy. Both require the winning players to use the modified mirror-reverse strategy. In the above game with three self-involutive edges it was necessary for Player~1 to mark one of them so that the remaining number of self-involutive edges would be even, which then allows Player~1 to implement the modified mirror-reverse strategy with their remaining moves. If the above graph was altered to have a cycle $C_8$ instead of the cycle $C_9$ then there would only be  two self-involutive edges and Player~2 could immediately respond to all of Player~1's moves using the mirror-reverse strategy. Figure~\ref{fig:c578played} shows a similar sample game to the one in Figure~\ref{fig:c579played}, where a cycle $C_8$ is used instead of the cycle $C_9$.

\begin{figure}[H]
\centering
\begin{tikzpicture}[scale=.6]
			 \begin{scope}[very thick, every node/.style={sloped,allow upside down}]
		     \Bvertex (A) at (0.9,-1){};
			 %\Bvertex (B) at (1.8, -1){};
			 \Bvertex (C) at (2.5,0.7){};
			 \Bvertex (D) at (3.8,2.2){};
			 \Bvertex (E) at (2.3,3.8){};
			 \Bvertex (F) at (.75,5){};
			 \Bvertex (G) at (-.7,3.8){};
			 \Bvertex (H) at (-2,2.2){};
			 \Bvertex (I) at (-0.8,0.7){};
		     \Bvertex (J) at (-3.2,3.5){};
			 \Bvertex (K) at (-5,3.5){};
			 \Bvertex (L) at (-5,0.5){};
			 \Bvertex (M) at (-3.2,0.5){};
			 \Bvertex (N) at (5.3,3.5){};
			 \Bvertex (S) at (5.3,.5){};
			 \Bvertex (O) at (6.8,3.5){};
			 \Bvertex (R) at (6.8,0.5){};
			 \Bvertex (P) at (8.5,3.5){};
			 \Bvertex (Q) at (8.5,0.5){};
			 \draw (A) to (C);
			 \draw (C) to (D); 
			 \draw (D) to (E);
			 \draw (E) to (F);
			 \draw (F) to (G); 
			 \draw (G) to (H); 
			 \draw (H) to (I);
			 \draw (I) to (A);
			 \draw (H) to (J); 
			 \draw (J) to (K); 
			 \draw (K) to (L); 
			 \draw (L) to (M);
			 \draw (M) to (H);
			 \draw (D) to (N); 
			 \draw (N) to (O);
			 \draw (O) to (P);
			 \draw (P) to (Q); 
			 \draw (Q) to (R); 
			 \draw (R) to (S);
			 \draw (S) to (D);
			 \draw[dashed] (.75,5) --(0.9,-1){};
			 \draw[dashed] (-1.8, 2.2) --(-5,2.2){};
			 \draw[dashed] (3.5,2.2) --(8.5,2.2){};
			 \draw (F) -- node[font=\small, label={[label distance=-5pt] below right:$1$} ] {\midarrow} (G);
			 \draw (E) -- node[font=\small, label={[label distance=-5pt]below left:$2$} ] {\midarrow} (F);
			 \draw (G) -- node[font=\small, label={[label distance=-7pt]:$3$} ] {\midarrow} (H);
			 \draw (D) -- node[font=\small, label={[label distance=-7pt]:$4$} ] {\midarrow} (E);
			 \draw (H) -- node[font=\small, label={[label distance=-6pt]:$5$} ] {\midarrow} (I);
			 \draw (C) -- node[font=\small, label={[label distance=-7pt]:$6$} ] {\midarrow} (D);
			 \draw (J) -- node[font=\small, label={[label distance=-7pt]:$7$} ] {\midarrow} (H);
			 \draw (H) -- node[font=\small, label={[label distance=-7pt]:$8$} ] {\midarrow} (M);
			 \draw (K) -- node[font=\small, label={[label distance=-5pt]:$9$} ] {\midarrow} (J);
			 \draw (P) -- node[font=\small, label={[label distance=-5pt]:$10$} ] {\midarrow} (Q);
			 \draw (O) -- node[font=\small, label={[label distance=-3pt]below:$11$} ] {\midarrow} (N);
			 \draw (S) -- node[font=\small, label={[label distance=-5pt]below:$12$} ] {\midarrow} (R);
			 \draw (N) -- node[font=\small, label={[label distance=-7pt]below:$13$} ] {\midarrow} (D);
			 \draw (D) -- node[font=\small, label={[label distance=-7pt]below:$14$} ] {\midarrow} (S);
			 \draw (A) -- node[font=\small, label={[label distance=-7pt]below:$15$} ] {\midarrow} (C);
			 \draw (I) -- node[font=\small, label={[label distance=-7pt]below:$16$} ] {\midarrow} (A);
			 
			 \end{scope}
\end{tikzpicture}

\caption{Completed game on the game board with an even number of self-involutive edges.}
\label{fig:c578played}
\end{figure}
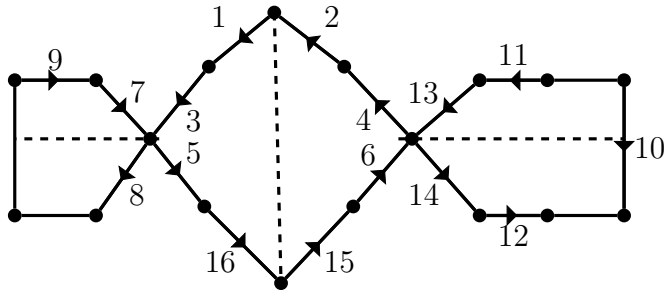

\subsection{Generalizing to cactus graphs with four or more cycles}
Now that we have shown how the modified mirror-reverse strategy works on cactus graphs consisting of  two cycles and three cycles, a natural question to ask is: Can the modified mirror-reverse strategy be applied to all triangle-free cactus graphs? 
Unfortunately, the answer is no.  Consider the example game shown in Figure~\ref{fig:c5976played} below.  In the example game we have labeled the problematic vertex ``$a$."

\begin{figure}[H]
\centering
\begin{tikzpicture}[scale=.6]
			 \begin{scope}[very thick, every node/.style={sloped,allow upside down}]
			 \Bvertex (A) at (0,-1.2){};
			 \Bvertex (B) at (1.8,-1.2){};
			 \Bvertex (C) at (2.5,0){};
			 \Bvertex (D) at (3.5,2.2)[label=above:$a$]{};
			 \Bvertex (E) at (2.3,3.9){};
			 \Bvertex (F) at (.75,4.9){};
			 \Bvertex (G) at (-.65,3.9){};
			 \Bvertex (H) at (-1.85,2.2){};
			 \Bvertex (I) at (-.9,0){};
		     \Bvertex (J) at (-2.7,3.2){};
			 \Bvertex (K) at (-4,3.2){};
			 \Bvertex (L) at (-4,1.2){};
			 \Bvertex (M) at (-2.7,1.2){};
			 \Bvertex (N) at (4.8,3.9){};
			 \Bvertex (O) at (6.4,3.9){};
			 \Bvertex (P) at (7.8,3.9){};
			 \Bvertex (Q) at (7.8,0.2){};
			 \Bvertex (R) at (6.4,0.2){};
			 \Bvertex (S) at (4.5,0.2){};
			 \Bvertex (T) at (5.4,-1){};
			 \Bvertex (U) at (5.4,-2.5){};
			 %\Bvertex (V) at (4.5,-3.2){};
			 \Bvertex (W) at (3.6,-2.5){};
			 \Bvertex (X) at (3.6,-1){};
			 \draw (A) to (B);
			 \draw (B) to (C);
			 \draw (C) to (D); 
			 \draw (D) to (E);
			 \draw (E) to (F);
			 \draw (F) to (G); 
			 \draw (G) to (H); 
			 \draw (H) to (I);
			 \draw (I) to (A);
			 \draw (H) to (J); 
			 \draw (J) to (K); 
			 \draw (K) to (L); 
			 \draw (L) to (M);
			 \draw (M) to (H);
			 \draw (D) to (N); 
			 \draw (N) to (O);
			 \draw (O) to (P);
			 \draw (P) to (Q); 
			 \draw (Q) to (R); 
			 \draw (R) to (S);
			 \draw (S) to (D);
			 \draw (S) to (T);
			 \draw (T) to (U); 
			 \draw (U) to (W);
			 %\draw (V) to (W);
			 \draw (W) to (X); 
			 \draw (X) to (S); 
			 \draw[dashed] (.75,4.9) --(0.9,-1.05){};
			 \draw[dashed] (-1.6, 2.2) --(-4,2.2){};
			 \draw[dashed] (4.1,1.25)--(7.6,3.7){};
			 \draw[dashed] (4.5,.5) --(4.5,-2.6){};
			 \draw (S) -- node[font=\small, label={[label distance=-7pt]:$1$} ] {\midarrow} (D);
			 \draw (B) -- node[font=\small, label={[label distance=-2pt]:$2$} ] {\midarrow} (A);
			 \draw (N) -- node[font=\small, label={[label distance=-5pt]:$3$} ] {\midarrow} (D);
			 \draw (S) -- node[font=\small, label={[label distance=-4pt]:$4$} ] {\midarrow} (R);
			 \draw (C) -- node[font=\small, label={[label distance=-6pt]:$5$} ] {\midarrow} (D);
			 \draw (H) -- node[font=\small, label={[label distance=-6pt]:$6$} ] {\midarrow} (I);
			 \draw (H) -- node[font=\small, label={[label distance=-6pt]below:$7$} ] {\midarrow} (G);
			 \draw (E) -- node[font=\small, label={[label distance=-6pt]below:$8$} ] {\midarrow} (D);
			 \end{scope}
			 
\end{tikzpicture}

\caption{Example of a game where the modified mirror reverse strategy fails.}
    \label{fig:c5976played}
\end{figure}
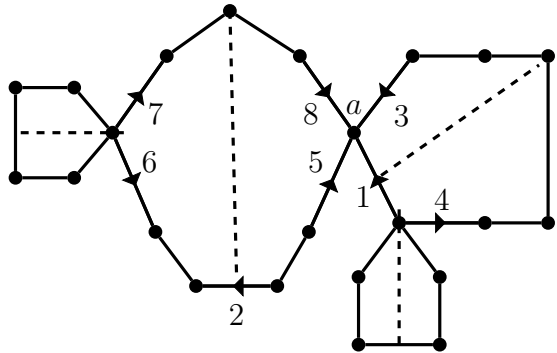

The even number of self-involutive edges suggests that Player~2 might win here by employing the modified mirror-reverse strategy. However, the strategy dictates an 8th move that would result in a sink at vertex $a$. Such a move is not permitted and thus our strategy cannot be employed. Note that this does not necessarily mean that Player~2 cannot win the game, only that they cannot win by employing the modified mirror-reverse strategy. 

Why does the modified mirror-reverse strategy fail in the case of four joined cycles but never in the case of two or three joined cycles? The sink/source issue is avoided in the case of the two or three joined cycles because each vertex of high degree has an axis of symmetry which passes through it. Thus, since each cycle is triangle-free, our strategy dictates that we can apply the mirror-reverse strategy to edges and their partner edges. Since applying the mirror-reverse strategy does not change the direction of any cycle, a sink or source is prevented at vertices of high degree. 
However in the case of four or more joined cycles in a cactus graph, it becomes possible for a vertex of high degree to not have any axis of symmetry passing through it, as seen at vertex $a$ in Figure~\ref{fig:c5976played}.

In Figure~\ref{fig:c59767} below, we resolve this issue by joining another cycle to our graph. The addition of the uppermost cycle $C_7$ shifts the axis of symmetry on the original cycle $C_7$ so that it now passes through vertex $a$.

\begin{figure}[H]
\centering
\begin{tikzpicture}[scale=.4]
			 \begin{scope}[very thick, every node/.style={sloped,allow upside down}]
			 \Bvertex (A) at (0,-1.2){};
			 \Bvertex (B) at (1.8,-1.2){};
			 \Bvertex (C) at (2.5,0){};
			 \Bvertex (D) at (3.5,2.2)[label=above:$a$]{};
			 \Bvertex (E) at (2.3,3.9){};
			 \Bvertex (F) at (.75,4.9){};
			 \Bvertex (G) at (-.65,3.9){};
			 \Bvertex (H) at (-1.85,2.2){};
			 \Bvertex (I) at (-.9,0){};
		     \Bvertex (J) at (-2.7,3.2){};
			 \Bvertex (K) at (-4,3.2){};
			 \Bvertex (L) at (-4,1.2){};
			 \Bvertex (M) at (-2.7,1.2){};
			 \Bvertex (N) at (4.8,3.9){};
			 \Bvertex (O) at (6.4,3.9){};
			 \Bvertex (P) at (7.8,3.9){};
			 \Bvertex (Q) at (7.8,0.2){};
			 \Bvertex (R) at (6.4,0.2){};
			 \Bvertex (S) at (4.5,0.2){};
			 \Bvertex (T) at (5.4,-1){};
			 \Bvertex (U) at (5.4,-2.5){};
			%\Bvertex (V) at (4.5,-3.2){};
			 \Bvertex (W) at (3.6,-2.5){};
			 \Bvertex (X) at (3.6,-1){};
			 \Bvertex (Y) at (6,4.8){};
			 \Bvertex (Z) at (6,6){};
			 \Bvertex (AA) at (6,7){};
			 \Bvertex (AB) at (3,7){};
			 \Bvertex (AC) at (3,6){};
			 \Bvertex (AD) at (3,4.8){};
			 \draw (A) to (B);
			 \draw (B) to (C);
			 \draw (C) to (D); 
			 \draw (D) to (E);
			 \draw (E) to (F);
			 \draw (F) to (G); 
			 \draw (G) to (H); 
			 \draw (H) to (I);
			 \draw (I) to (A);
			 \draw (H) to (J); 
			 \draw (J) to (K); 
			 \draw (K) to (L); 
			 \draw (L) to (M);
			 \draw (M) to (H);
			 \draw (D) to (N); 
			 \draw (N) to (O);
			 \draw (O) to (P);
			 \draw (P) to (Q); 
			 \draw (Q) to (R); 
			 \draw (R) to (S);
			 \draw (S) to (D);
			 \draw (S) to (T);
			 \draw (T) to (U); 
			 \draw (U) to (W);
			 %\draw (V) to (W);
			 \draw (W) to (X); 
			 \draw (X) to (S); 
			 \draw (N) to (Y); 
			 \draw (Y) to (Z); 
			 \draw (Z) to (AA);
			 \draw (AA) to (AB);
			 \draw (AB) to (AC);
			 \draw (AC) to (AD); 
			 \draw (AD) to (N);
			 \draw[dashed] (.75,4.9) --(0.9,-1.05){};
			 \draw[dashed] (-1.6, 2.2) --(-4,2.2){};
			 \draw[dashed] (3.5,2.2)--(7.8,2.2){};
			 \draw[dashed] (4.5,.5) --(4.5,-3){};
			 \draw[dashed] (4.7,3.9) --(4.7,7.2){};
			 \end{scope}
\end{tikzpicture}

\caption{The game board in Figure~\ref{fig:c5976played} modified to ensure that every degree 4 vertex has an axis of symmetry passing through it.}
\label{fig:c59767}
\end{figure}
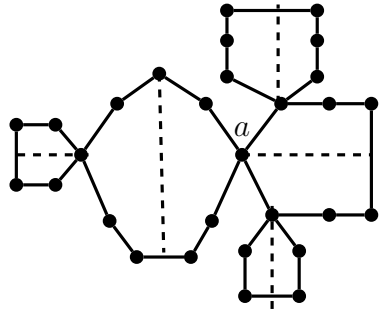

Although the modified mirror-reverse strategy cannot be successfully applied to every triangle-free cactus graph in which every edge belongs to exactly one cycle, it can be applied to any such graph in which every cycle has an axis of reflective symmetry such that
\begin{itemize}
    \item symmetry partners vertices of degree 2 with one another and vertices of high degree with other vertices of high degree (possibly themselves), and
    \item every vertex of high degree has at least one axis of symmetry passing through it.
\end{itemize}
There are many types of graphs that satisfy these criteria. Some examples are shown in Figure~\ref{fig:exampleboards}.

 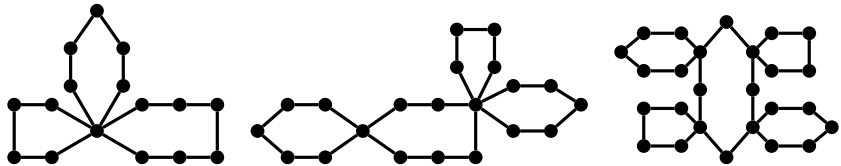
\begin{figure}[H]\center
\[
\begin{array}{ccc}	   
\begin{tikzpicture}[scale=.5]
\begin{scope}[very thick, every node/.style={sloped,allow upside down}]
    \Bvertex (A) at (0,0){};
	\Bvertex (B) at (-1.2,-.7){};
	\Bvertex (C) at (-2.2, -.7){};
	\Bvertex (D) at (-2.2, .7){};
    \Bvertex (E) at (-1.2, .7){};
	\Bvertex (F) at (-.7, 1.2){};
	\Bvertex (G) at (-.7, 2.2){};
	\Bvertex (H) at (0, 3.2){};
	\Bvertex (I) at (.7, 2.2){};
    \Bvertex (J) at (.7,1.2){};
	\Bvertex (K) at (1.2,.7){};
	\Bvertex (L) at (2.2,.7){};
	\Bvertex (M) at (3.2,.7){};
	\Bvertex (N) at (3.2,-.7){};
	\Bvertex (O) at (2.2,-.7){};
    \Bvertex (P) at (1.2,-.7){};
    \draw (A) to (B);
    \draw (B) to (C);
    \draw (C) to (D);
    \draw (D) to (E);
    \draw (E) to (A);
    \draw (A) to (F);
    \draw (F) to (G);
    \draw (G) to (H);
    \draw (H) to (I);
    \draw (I) to (J);
    \draw (J) to (A);
    \draw (A) to (K);
    \draw (K) to (L);
    \draw (L) to (M);
    \draw (M) to (N);
    \draw (N) to (O);
    \draw (O) to (P);
    \draw (P) to (A);
\end{scope}
\end{tikzpicture}
&

\begin{tikzpicture}[scale=.5]
\begin{scope}[very thick, every node/.style={sloped,allow upside down}]
    \Bvertex (A) at (0,0){};
	\Bvertex (B) at (-1, .7){};
	\Bvertex (C) at (-2, .7){};
	\Bvertex (D) at (-2.8, 0){};
    \Bvertex (E) at (-2, -.7){};
	\Bvertex (F) at (-1, -.7){};
	\Bvertex (G) at (1, -.7){};
	\Bvertex (H) at (2, -.7){};
	\Bvertex (I) at (3, -.7){};
    \Bvertex (J) at (3, .7){};
	\Bvertex (K) at (2,.7){};
	\Bvertex (L) at (1,.7){};
	\Bvertex (M) at (2.5,1.7){};
	\Bvertex (N) at (2.5, 2.7){};
	\Bvertex (O) at (3.5, 2.7){};
    \Bvertex (P) at (3.5, 1.7){};
    \Bvertex (Q) at (4, 1.2){};
	\Bvertex (R) at (5, 1.2){};
	\Bvertex (S) at (5.8, .7){};
	\Bvertex (T) at (5, 0){};
	\Bvertex (U) at (4, 0){};
	\draw (A) to (B);
    \draw (B) to (C);
    \draw (C) to (D);
    \draw (D) to (E);
    \draw (E) to (F);
    \draw (F) to (A);
    \draw (A) to (G);
    \draw (G) to (H);
    \draw (H) to (I);
    \draw (I) to (J);
    \draw (J) to (K);
    \draw (K) to (L);
    \draw (L) to (A);
    \draw (J) to (M);
    \draw (M) to (N);
    \draw (N) to (O);
    \draw (O) to (P);
    \draw (P) to (J);
    \draw (J) to (Q);
    \draw (Q) to (R);
    \draw (R) to (S);
    \draw (S) to (T);
    \draw (T) to (U);
    \draw (U) to (J);
\end{scope}
\end{tikzpicture}  
&

 \begin{tikzpicture}[scale=.5]
\begin{scope}[very thick, every node/.style={sloped,allow upside down}]
    \Bvertex (A) at (0,1.8){};
	\Bvertex (B) at (.7,1){};
	\Bvertex (C) at (.7, 0){};
	\Bvertex (D) at (.7, -1){};
    \Bvertex (E) at (0,-1.8){};
	\Bvertex (F) at (-.7, -1){};
	\Bvertex (G) at (-.7, 0){};
	\Bvertex (H) at (-.7, 1){};
	\Bvertex (I) at (-1.2, 1.5){};
    \Bvertex (J) at (-2.2, 1.5){};
	\Bvertex (K) at (-2.8, 1){};
	\Bvertex (L) at (-2.2, .5){};
	\Bvertex (M) at (-1.2, .5){};
	\Bvertex (N) at (-1.2, -.5){};
	\Bvertex (O) at (-2.2, -.5){};
    \Bvertex (P) at (-2.2, -1.5){};
    \Bvertex (Q) at (-1.2, -1.5){};
	\Bvertex (R) at (1.2, -1.5){};
	\Bvertex (S) at (2.2, -1.5){};
	\Bvertex (T) at (2.8, -1){};
	\Bvertex (U) at (2.2, -.5){};
	\Bvertex (V) at (1.2, -.5){};
	\Bvertex (W) at (1.2, .5){};
	\Bvertex (X) at (2.2, .5){};
	\Bvertex (Y) at (2.2, 1.5){};
	\Bvertex (Z) at (1.2, 1.5){};
	\draw (A) to (B);
    \draw (B) to (C);
    \draw (C) to (D);
    \draw (D) to (E);
    \draw (E) to (F);
    \draw (F) to (G);
    \draw (G) to (H);
    \draw (H) to (A);
    \draw (H) to (I);
    \draw (I) to (J);
    \draw (J) to (K);
    \draw (K) to (L);
    \draw (L) to (M);
    \draw (M) to (H);
    \draw (F) to (N);
    \draw (N) to (O);
    \draw (O) to (P);
    \draw (P) to (Q);
    \draw (Q) to (F);
    \draw (D) to (R);
    \draw (R) to (S);
    \draw (S) to (T);
    \draw (T) to (U);
    \draw (U) to (V);
    \draw (V) to (D);
    \draw (B) to (W);
    \draw (W) to (X);
    \draw (X) to (Y);
    \draw (Y) to (Z);
    \draw (Z) to (B);
\end{scope}
\end{tikzpicture}\\

\end{array}
\]
\caption{Examples of game boards that satisfy the criteria needed to use the modified mirror-reverse strategy. }\label{fig:exampleboards}
\end{figure}

\section{Main result for cactus graphs}\label{mainresultsection}
In this section, we state and prove our general result for triangle-free cactus graphs.

\begin{theorem}\label{mainresult}
Let $G$ be a triangle-free cactus graph in which every edge belongs to exactly one cycle. Assume that $G$ has a fixed set of symmetries satisfying the following properties:
\begin{enumerate}
\item Each cycle has an axis of reflective symmetry under which every vertex of degree at least 4 is partnered via symmetry with another vertex of degree at least 4 (possibly itself).

\item For each vertex of degree at least 4 there is at least one of these axes of symmetry which passes through it.
\end{enumerate}

If the total number of self-involutive edges in $G$ is even, then Player~2 can win by employing the modified mirror-reverse strategy. If the total number of self-involutive edges in $G$ is odd, then Player~1 can win by first marking a self-involutive edge and then employing the modified mirror-reverse strategy.
\end{theorem}

\begin{proof}
Note that in the case when the number of self-involutive edges is odd, after marking the first self-involutive edge Player~1 essentially functions as Player~2 would in the case in which there are an even number of self-involutive edges. An example of such a game is shown in Figure~\ref{fig:c579played}. Thus, it suffices to consider only one case. In the style of Alvarado et al.~\cite{Alvarado}, we will call the player with the winning strategy Player W. Here, Player W will be using the modified mirror-reverse strategy to respond to the moves of their opponent, who we will call Player X. If $ij$ is the edge with endpoints $i$ and $j$, playing $i\rightarrow j$ means that a player will mark an arrow from $i$ to $j$. Let $i'$ and $j'$ denote the partner vertices of $i$ and $j$ respectively under the symmetry. Under the mirror-reverse strategy Player~W would complete a cycle if possible, but otherwise would look at the edge  Player~X played on the previous move ($ i \rightarrow j$) and play the \emph{mirror-reverse} ($j'\rightarrow i'$).

We can assume that the current number of unmarked self-involutive edges is even. Since there is an even number of self-involutive edges, we can think of them as being ``paired'' by the strategy even before gameplay begins. That is, when Player X plays on a certain self-involutive edge Player W responds by playing on its paired self-involutive edge. Such a response is demonstrated in moves 1 and 2 of the game played in Figure~\ref{fig:c5976played}. And when Player X plays on a non-self-involutive edge but Player W cannot respond by playing its mirror-reverse due to it being a death move, Player W responds by playing on the self-involutive edge that is paired with the self-involutive edge belonging to the cycle in which the death move would have occurred.  Such a response is demonstrated in moves 9 and 10 of the game played in Figure~\ref{fig:c578played}.

To prove that Player W indeed has the winning strategy we must show that the move dictated by the strategy is always available (previously unmarked), will never lead to a sink or source for Player W, and will never lead to a death move for Player W. Thus since Player X can never complete a cycle or make the final move in the game, Player W wins.

First we will show that the move dictated by the modified mirror-reverse strategy is always available for Player W. If the strategy dictates that Player W play on an edge that is not self-involutive, then Player X must have just played on its mirror-reverse edge. Clearly the edge is available for Player W to mark because had this edge been marked previously, either by Player X or Player W, it's mirror-reverse would have already been marked as well and hence unavailable for Player X to be marking now. 

The case in which the strategy dictates that Player W mark a self-involutive edge is more complicated. This could occur because Player X marked another self-involutive edge or because Player X marked an edge which is not self-involutive, but whose mirror-reverse would be a death move. (Again, examples of each of these situations can be seen in Figures~\ref{fig:c5976played} and \ref{fig:c578played} respectively.) Note that this second case could only occur on a cycle with exactly one unmarked self-involutive edge.
This is because a death move on a cycle only occurs on the second-to-last unmarked edge. If the cycle were even with either two or zero unmarked self-involutive edges, then since the remaining edges are marked in pairs via mirror-reverse, parity dictates that the second-to-last move would belong to Player X. The same is true for an odd cycle in which the single self-involutive edge had already been marked.

Thus, if Player W cannot play a mirror-reverse of a non-self-involutive edge due to it being a death move, then one of the two remaining unmarked edges must have been self-involutive. Marking this self-involutive edge would now be a death move itself, so even though it is unmarked this edge is now considered currently unplayable.  We observe this in the cycle $C_5$ of the example game in Figure~\ref{fig:c578played} when Player~2 must use the strategy to respond to the 9th move of the game. So we make the important observation that the modified mirror-reverse strategy causes self-involutive edges to become currently unplayable (either by being marked or becoming death moves) in pairs, with Player W playing immediately after Player X. 
That is, whenever the strategy dictates that Player W play on a self-involutive edge, it must be the case that Player X just played in such a way to make its paired edge no longer currently playable. 
Thus, since there are an even number of self-involutive edges, it cannot be Player X that marks the last playable self-involutive edge. So Player W's move is always available.

Next we will show that the modified mirror-reverse strategy will never dictate that Player W play in a way that will create a sink or source at a vertex. For high degree vertices, a sink/source is prevented by Property 2 of graph $G$ described in the theorem. Since such a vertex must have an axis of reflective symmetry passing through it and $G$ is triangle-free, this vertex would be incident to a pair of edges being marked via mirror-reverse. Thus, if one edge enters the vertex its partner exits the vertex and this prevents a sink/source at the vertex.  For example, consider the degree 4 vertices of the game board shown in Figure~\ref{fig:c578played}. Moves 7 and 8 and moves 13 and 14, respectively, demonstrate how the strategy prevents a sink/source at high degree vertices. Let us instead consider vertices of degree 2. In the case that the two incident edges are both not self-involutive, they will be marked by Player W only by employing the mirror-reverse strategy. By Property 1 of graph $G$, degree 2 vertices are paired via symmetry with other degree 2 vertices, thus a sink/source cannot be created by Player W without a sink/source having just been created by Player X at the partner vertex. In the case in which one of the incident edges is self-involutive, note that this self-involutive edge must be incident to two other edges which are not self-involutive and are in fact paired with one another under symmetry. Thus, mirror-reversing ensures that these edges can only be marked with the same direction (clockwise or counterclockwise). Since there are no rules about the direction which Player W must mark a self-involutive edge, Player W can choose the direction which avoids a sink/source incident to a self-involutive edge. Note that because of the symmetry and mirror-reverse play, the edge cannot be unmarkable.

Finally, we will show that the modified mirror-reverse strategy can never lead to Player W making a death move. Since the strategy dictates that Player W only mark the mirror-reverse when it is not a death move, we need only consider the case in which Player W marks a self-involutive edge. However, as noted before, self-involutive edges become currently unplayable in pairs, with Player W playing second. So if a self-involutive edge were a death move for Player W, then it was already currently unplayable and hence its partner edge would have also been currently unplayable---that is, it was either already marked, and hence unable to be played by Player X, or it was a death move for Player X, in which case the strategy would dictate that Player W complete the cycle and win.
\end{proof}

The careful reader will observe that although the above theorem requires that we fix a set of symmetries on our graph $G$, it is possible that multiple sets of such symmetries exist. We note that our choice of symmetries does not matter, as long as Properties 1 and 2 of the theorem hold.  Since every odd cycle will have an axis of symmetry with exactly one self-involutive edge and every even cycle will have an axis of symmetry with either two or zero self-involutive edges, changing the set of symmetries on $G$ cannot affect the parity of self-involutive edges. So although the moves dictated by the strategy may differ, the winning player is fixed regardless of the set of symmetries used.

\section{Cactus graphs that are not triangle-free}\label{trianglecactus}
Recall that in Theorem~\ref{mainresult} we restricted our cactus graphs to be triangle-free. A natural question regarding this restriction is: Is there a winning strategy for any player if the cactus graph is not triangle-free? We answer this question for cactus graphs in which two cycles are joined together by a vertex with one cycle being a $C_3$ graph in the following result. This result also provides an example in which there is an even number of edges in the board and Player~1 has a winning strategy, thus answering the question of Alvarado et al.~\cite{Alvarado} mentioned at the start of Section~\ref{cactusgraph}.

\begin{theorem}\label{triangletheorem}
Let $G$ be a board containing two cycles, $C_3$ and $C_n$, with $n\geq4$, such that $C_3$ and $C_n$ are connected by a single vertex, denoted vertex $a$, and every edge belongs to exactly one cycle. Player~1 has a winning strategy on $G$.
\end{theorem}

\begin{proof}
In the case where $n$ is even, $G$ has reflective symmetry by drawing the axis of symmetry to only intersect the degree 4 vertex as shown in Figure~\ref{fig:symgraph} Board $A$ and there is exactly one self-involutive edge whose vertices are not fixed by the reflection. Therefore Corollary~\ref{mirrorsymbasic} holds and Player~1 has a winning strategy by marking the self-involutive edge on their first move and then using the mirror-reverse strategy for subsequent moves. 

In the case where $n$ is odd, Player~1's strategy is to ensure that vertex $a$ is an almost-sink. Note that a similar strategy for creating an almost-source at vertex $a$ would also produce a win for Player~1.  On Player~1's first move, they will mark one of the edges on the cycle $C_3$ that is incident to vertex $a$ with an arrow towards vertex $a$. Player~1's second move is dependent upon the first move of Player~2. There are three possible moves which Player~2 could make that are not death moves. In what follows we describe the strategy for Player~1 in each of the three cases. For examples of game play illustrating these three cases we have provided sample games on the board consisting of a cycle $C_3$ and a cycle $C_5$ joined at a single vertex in Figure~\ref{fig:C3C5graph} below.

In case one, if Player~2's first move is on the cycle $C_3$ then they would play on the other edge incident to vertex $a$ in the direction towards vertex $a$ to avoid a death move. In this case, Player~1's second move would be to play on the cycle $C_n$ on one of the edges incident to vertex $a$ with an arrow towards vertex $a$. This creates an almost sink at vertex $a$.
 
In case two, if Player~2's first move is on the cycle $C_n$ and they mark an edge that is incident to any vertex within distance one of vertex $a$, then Player~1's second move should be to mark the edge paired with that one via symmetry in the direction opposite  Player~2's first move (that is, Player~1 should use the mirror-reverse strategy). Then on Player~1's third turn, they should ensure that the remaining edge incident vertex $a$ on the cycle $C_3$ is marked with an arrow towards vertex $a$ by either marking it or observing that Player~2 marked it on their second move, which did not result in a death move. Note that if an almost sink has not yet been created, Player~1 can do so on their fourth move.
 
In case three, if Player~2's first move is not incident to a vertex within distance one of vertex $a$, then Player~1's second move is to mark the remaining edge in the cycle $C_3$ that is incident to vertex $a$ with an arrow towards vertex $a$, which leaves the remaining unmarked edge in $C_3$ unmarkable.  After Player~2's second move, Player~1 should either observe that an almost-sink has been created at vertex $a$ or play an edge incident to vertex $a$ on the cycle $C_n$ to create one.

In all, this strategy ensures the creation of an unmarkable edge in the cycle $C_3$ and that the number of unmarkable edges on $C_n$ are even. Hence, there are an odd number of edges that can be marked on the board $G$ and this leads to a Player~1 win. 

\end{proof}

\begin{figure}[ht]
\[
\begin{array}{ccc}	
\begin{tikzpicture}[scale=.6]
		 \begin{scope}[very thick, every node/.style={sloped,allow upside down}]
			 \Bvertex (A) at (1.5,0){};
			 \Bvertex (B) at (-.5, 1.414){};
			 \Bvertex (E) at (-.5,-1.414){};
			 \Bvertex (F) at (3.2, 1.414){};
			 \Bvertex (H) at (5.414,1.414){};
			 \Bvertex (I) at (5.414,-1.414){};
			 \Bvertex (K) at (3.2,-1.414){};
			 \draw (A) to (B);
			 \draw (B) to (E);
			 \draw (E) to (A);
			 \draw (A) to (F);
			 \draw (F) to (H); 
			 \draw (H) to (I); 
			 \draw (I) to (K);	
			 \draw (K) to (A); 
			 \draw (B) -- node[font=\small, label={[label distance=-7pt]:$1$} ] {\midarrow} (A);
			 \draw (E) -- node[font=\small, label={[label distance=-9pt]below:$2$} ] {\midarrow} (A);
			 \draw (F) -- node[font=\small, label={[label distance=-9pt] below right:$3$} ] {\midarrow} (A);
			 \draw (I) -- node[font=\small, label={[label distance=-4pt]below:$4$} ] {\midarrow} (K);
			 \draw (H) -- node[font=\small, label={[label distance=-4pt]below:$5$} ] {\midarrow} (F);
			 \end{scope}
\end{tikzpicture}
&
\begin{tikzpicture}[scale=.6]
			 \begin{scope}[very thick, every node/.style={sloped,allow upside down}]
			 \Bvertex (A) at (1.5,0){};
			 \Bvertex (B) at (-.5, 1.414){};
			 \Bvertex (E) at (-.5,-1.414){};
			 \Bvertex (F) at (3.2, 1.414){};
			 \Bvertex (H) at (5.414,1.414){};
			 \Bvertex (I) at (5.414,-1.414){};
			 \Bvertex (K) at (3.2,-1.414){};
			 \draw (A) to (B);
			 \draw (B) to (E);
			 \draw (E) to (A);
			 \draw (A) to (F);
			 \draw (F) to (H); 
			 \draw (H) to (I); 
			 \draw (I) to (K);	
			 \draw (K) to (A); 
			 \draw (B) -- node[font=\small, label={[label distance=-7pt]:$1$} ] {\midarrow} (A);
			 \draw (A) -- node[font=\small, label={[label distance=-6pt]:$2$} ] {\midarrow} (F);
			 \draw (K) -- node[font=\small, label={[label distance=-9pt]:$3$} ] {\midarrow} (A);
			 \draw (H) -- node[font=\small, label={[label distance=-5pt]:$4$} ] {\midarrow} (I);
			 \draw (E) -- node[font=\small, label={[label distance=-9pt]below:$5$} ] {\midarrow} (A);
			 	 \draw (I) -- node[font=\small, label={[label distance=-4pt]below:$6$} ] {\midarrow} (K);
			 \draw (F) -- node[font=\small, label={[label distance=-4pt]above:$7$} ] {\midarrow} (H);
			 \end{scope}
\end{tikzpicture}
&
\begin{tikzpicture}[scale=.6]
			 \begin{scope}[very thick, every node/.style={sloped,allow upside down}]
			 \Bvertex (A) at (1.5,0){};
			 \Bvertex (B) at (-.5, 1.414){};
			 \Bvertex (E) at (-.5,-1.414){};
			 \Bvertex (F) at (3.2, 1.414){};
			 \Bvertex (H) at (5.414,1.414){};
			 \Bvertex (I) at (5.414,-1.414){};
			 \Bvertex (K) at (3.2,-1.414){};
			 \draw (A) to (B);
			 \draw (B) to (E);
			 \draw (E) to (A);
			 \draw (A) to (F);
			 \draw (F) to (H); 
			 \draw (H) to (I); 
			 \draw (I) to (K);	
			 \draw (K) to (A); 
			 \draw (B) -- node[font=\small, label={[label distance=-7pt]:$1$} ] {\midarrow} (A);
			 \draw (I) -- node[font=\small, label={[label distance=-4pt]below:$2$} ] {\midarrow} (H);
			 \draw (E) -- node[font=\small, label={[label distance=-7pt]below:$3$} ] {\midarrow} (A);
			 \draw (A) -- node[font=\small, label={[label distance=-4pt]:$4$} ] {\midarrow} (F);
			 \draw (K) -- node[font=\small, label={[label distance=-7pt]:$5$} ] {\midarrow} (A);
			 \end{scope}
\end{tikzpicture}\\
\text{Case } 1 & \text{Case }  2& \text{Case }  3\\
\end{array}
\]
\caption{An example showing of cycles $C_3$ and $C_5$ joined at a vertex displaying the three possible winning strategies for Player~1 in Theorem~\ref{triangletheorem}.}
\label{fig:C3C5graph}
\end{figure}
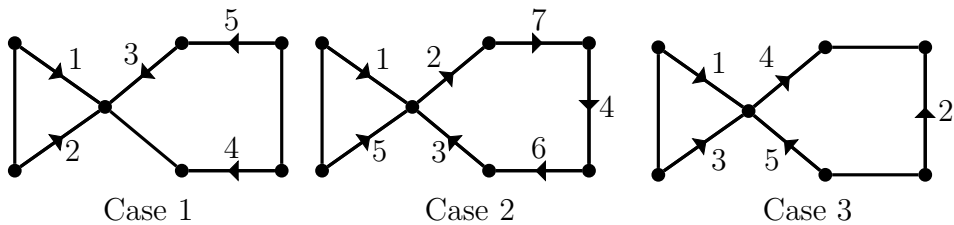

\section{Further Questions}
In this paper, we found a winning strategy for triangle-free cactus graphs with certain properties. However, there are questions that are unanswered from our studies: 
 
\begin{itemize}
    \item In our main result, we required our graphs to be triangle-free cactus graphs and to have a fixed set of symmetries that satisfied two properties (See Theorem~\ref{mainresult}). If we relax these criteria, can we determine a  strategy for cactus graphs without these properties? 
    \item The  modified mirror-reverse strategy was used for certain types of cactus graphs in our results.  Are there other classes of graphs that the modified mirror-reverse strategy can be applied to? 
    \item Can we utilized the modified mirror-reverse strategy on graphs with other types of symmetry (not just reflective symmetry)? 
    \item  We noted that our Theorem~\ref{triangletheorem} answers a question of Alvarado et at.~\cite{Alvarado} in the negative, by leading to examples of game boards with an even number of edges for which Player~1 possesses a winning strategy. However, such examples will all include a 3-cycle. We might ask the following variation of the question in \cite{Alvarado}: If we consider only triangle-free game boards, does it then follow that if the number of edges in the board is odd, Player 1 has a winning strategy, and otherwise Player 2 has a winning strategy?
    
\end{itemize}

\section{Acknowledgements}
We extend our gratitude to Jonah Amundsen and to Peter Graziano, doctoral student at the University of Connecticut, for their assistance with early aspects of the project.

The authors would like to thank the reviewers for their comments in reviewing the paper. We also thank the University of Wisconsin-Eau Claire Department of Mathematics, the Office of Research and Sponsored Projects for supporting Jonah Amundsen, Heather Baranek, and Shanise Walker on this project. In addition, we also thank the University of Wisconsin-Eau Claire Foundation, Walter M. Reid First Year Research Fellowship, and the Blugold Fellowship for supporting Heather Baranek. Most of the work for this project was completed while Samuel Adefiyiju and Alison LaBarre were students at Providence College.

\noindent {\bf Samuel Adefiyiju} is a Software Engineer at Raytheon Missiles and Defense who graduated Summa Cum Laude from Providence College with a  Bachelor's degree in Computer Science and a minor in Mathematics. His interests are software development, cybersecurity, algorithm optimization, and database management. While pursuing his undergraduate degree, Mr. Adefiyiju gained significant industry experience interning at the Rhode Island Department of Transportation, Ernst and Young, and Raytheon Missiles and Defense. Mr. Adefiyiju’s passion for both mathematics and computer science ultimately allowed him to find a home in the software development industry where he continues to contribute to providing high quality software to various clients.   \\

\noindent {\bf Heather Baranek} will receive her Bachelor's of Science in Software Engineering and Statistics from the University of Wisconsin- Eau Claire in 2022. As an undergraduate, she joined the Game of Cycles research team due to her interest in graph theory and its applicability in computer science.  \\

\noindent {\bf Abigail Daly} is currently a junior at Providence College where she is pursuing a degree in Computer Science with a minor in Mathematics. She was drawn to the project when hearing a lecture about Game Theory in her discrete math class. She currently works as a Student Field Technician for Providence College, and after graduating hopes to pursue a career in cybersecurity.  \\

\noindent {\bf Xadia M. Goncalves} is a current senior undergraduate Biology major and Spanish minor at Providence College. She is also working on an independent research project that focuses on predator-prey interactions. She is a Western University of Health Sciences Summer Health Professions Education Program (SHPEP) alumna and Resident Peer Assistant where she developed skills through clinical exposure and gained experience in health disparities. She also works closely with the Providence College admissions office as a Diversity Outreach Team member to recruit students from underrepresented backgrounds.  \\

\noindent {\bf Mary Leah Karker} received her Ph.D. in mathematics from Wesleyan University in 2016. She held a visiting position at Connecticut College for two years before joining the faculty at Providence College where she enjoys teaching a variety of undergraduate courses and mentoring students in research. While she loves exploring new areas of mathematics, her primary research interests are in mathematical logic and foundations. \\

\noindent {\bf Alison LaBarre} received her Bachelor of Arts in mathematics with a minor in computer science from Providence College. After graduating she was offered a job at Newgrange Design, where she is in the process of becoming a printed circuit board designer. \\

\noindent {\bf Shanise Walker} is an Assistant Professor of Mathematics at the University of Wisconsin-Eau Claire. She received her Ph.D. from Iowa State University in 2018. Her research interests lie in combinatorics and graph theory. She has supervised several undergraduate research projects. Walker is active in service to the mathematical profession related to equity, diversity, inclusion, and belonging.
\end{document}